\documentclass[11pt]{amsart}

\usepackage{amsmath}
\usepackage{amssymb}
\usepackage{graphicx}
\usepackage{soul}
\usepackage[utf8]{inputenc}
\usepackage{mathtools}
\usepackage{amssymb}
\usepackage{color}
\usepackage{hyperref}
\usepackage{cleveref}

\newtheorem{theorem}{Theorem}[section]

\newtheorem{lemma}[theorem]{Lemma}
\newtheorem{proposition}[theorem]{Proposition}
\theoremstyle{definition}
\newtheorem{definition}[theorem]{Definition}

\newtheorem{remark}[theorem]{Remark}

\numberwithin{equation}{section}

\newcommand{\Lip}{\operatorname{Lip}}

\newcommand{\dom}{\operatorname{dom}}

\newcommand{\ind}{\operatorname{ind}}
\newcommand{\im}{\operatorname{im}}

\newcommand{\inti}{\operatorname{interior}}
\newcommand{\graph}{\operatorname{gr}}
\newcommand{\tn}{\mathbb{P}}
\newcommand{\calD}{\mathcal{D}}

\newcommand{\dist}{\operatorname{dist}}

\newcommand{\Hd}{\dim_{\mathrm{H}}}
\newcommand{\abs}[1]							
	{\left| #1 \right|}

\renewcommand{\d}{\, \mathrm{d}}
\renewcommand{\P}{\mathbb P}
\newcommand{\M}{\mathcal M}

\def\XXint#1#2#3{{\setbox0=\hbox{$#1{#2#3}{\int}$ }
		\vcenter{\hbox{$#2#3$ }}\kern-.6\wd0}}

\newcommand{\N}{\mathbb N}
\newcommand{\R}{\mathbb R}

\renewcommand{\H}{\mathcal H}
\renewcommand{\L}{\mathcal L}
\renewcommand{\S}{\mathbb S^{n-1}}

\begin{document}

\date{}
\author{David Bate}
\email{david.bate@warwick.ac.uk}
\address{Zeeman Building, University of Warwick, CV4 7AL, UK}
\author{Ilmari Kangasniemi}
\email{kikangas@syr.edu}
\address{215 Carnegie Building, Syracuse University, NY 13244, USA}
\author{Tuomas Orponen}
\email{tuomas.t.orponen@jyu.fi}
\address{P.O. Box 35 (MaD), FI-40014, University of Jyväskylä, Finland}

\thanks{D.B. is supported by the Academy of Finland via the project \emph{Projections, densities and rectifiability: new settings for classical ideas}, grant No. 308510. IK is supported by the doctoral program DOMAST of the University of Helsinki. T.O. is supported by the Academy of Finland via the project \emph{Quantitative rectifiability in Euclidean and non-Euclidean spaces}, grant No. 309365.”}

\begin{abstract}
Suppose that $(X,d,\mu)$ is a metric measure space of finite Hausdorff dimension and that, for every Lipschitz $f \colon X \to \R$, $\Lip(f,\cdot)$ is dominated by every upper gradient of $f$. We show that $X$ is a Lipschitz differentiability space, and the differentiable structure of $X$ has dimension at most $\Hd X$. Since our assumptions are satisfied whenever $X$ is doubling and satisfies a Poincar\'e inequality, we thus obtain a new proof of Cheeger's generalisation of Rademacher's theorem. 

Our approach uses Guth's multilinear Kakeya inequality for neighbourhoods of Lipschitz graphs to show that any non-trivial measure with $n$ independent Alberti representations has Hausdorff dimension at least $n$.
\end{abstract}

\title[Cheeger's theorem via the multilinear Kakeya inequality]{Cheeger's differentiation theorem via the multilinear Kakeya inequality}

\keywords{Alberti representations, differentiation of Lipschitz functions, multilinear Kakeya inequality}
\subjclass[2010]{30L99 (Primary), 28A50, 49Q15 (Secondary)}
\maketitle

\section{Introduction}
Throughout this article, a \emph{metric measure space} $(X,d,\mu)$ will refer to a complete and separable metric space $(X,d)$ equipped with a $\sigma$-finite Borel measure $\mu$.

Rademacher's theorem states that any Lipschitz $f\colon \R^n\to \R$ is differentiable almost everywhere. Cheeger \cite{cheeger-diff} gave a far reaching generalisation of Rademacher's theorem to any \emph{PI-space}, that is, a doubling metric measure space satisfying a Poincar\'e inequality.
One of the novelties of Cheeger's work is the introduction of a notion of a derivative in the setting of metric spaces:
\begin{definition}
Fix $x_{0}\in X$ and, for $n\geq 0$ an integer, fix a Lipschitz $\phi\colon X \to \R^n$.
A function $f\colon X \to \R$ is \emph{differentiable with respect to $\phi$} at $x_0$ if there exists a unique, linear $Df(x_0) \colon\R^n\to \R$ such that
	\[ \lim_{x\to x_0} \frac{|f(x)-f(x_0)-Df(x_0)(\phi(x)-\phi(x_0))|}{d(x,x_0)}=0.\]
\end{definition}

With a notion of a derivative established, it is easy to define what it means for a metric measure space $(X,d,\mu)$ to satisfy a generalisation of Rademacher's theorem:
\begin{definition}\label{lipDiffSpace}
A metric measure space $(X,d,\mu)$ is a \emph{Lipschitz differentiability space} if there exist a countable Borel decomposition $X=\cup_i U_i$ and countably many Lipschitz $\phi_i\colon X \to \R^{n_i}$, $i \in \N$, such that,
for every Lipschitz $f\colon X \to \R$ and every $i\in\N$, $f$ is differentiable at $\mu$-a.e.\ point in $U_i$ with respect to $\phi_i$.
\end{definition}
\begin{theorem}[Cheeger]\label{cheeger}
 Every PI space is a Lipschitz differentiability space.
\end{theorem}

One remarkable consequence of this theory is the non-biLipschitz embeddability of PI spaces into a large class of Banach spaces, and the relationship to the sparsest cut problem. See \cite{cheegerkleiner-rnp,MR2892612,lee2006lp,MR2630066}.

After the work of Cheeger, Keith \cite{keith} gave a second proof of Cheeger's theorem, clarifying many of Cheeger's ideas, by studying a ``Lip-lip'' inequality for Lipschitz functions on $X$.
Following this, Kleiner--Mackay \cite{kleinermackay-diffstructures} gave a refinement of the arguments of Cheeger and Keith, giving an accessible account of the results and proof strategy.

Sometime later, the first author \cite{Bate_2014} gave several characterisations of Lipschitz differentiability spaces in terms of a rich structure of rectifiable curves, known as an \emph{Alberti representation} of $\mu$ - a collection of rectifiable curves that gives rise to a \emph{partial derivative} of any Lipschitz function on $X$ at $\mu$ almost every point (see \Cref{def:alberti-representation}).
Assuming that $\mu$ has many \emph{independent} Alberti representations (see \Cref{def:independent}), the partial derivatives obtained from each Alberti representation can be combined to form the Cheeger derivative, similarly to how the gradient on Euclidean space consists of partial derivatives.

Alberti representations first appear in the work of Alberti \cite{MR1215412} on the rank one property of functions of bounded variation.
They later appeared in the work of Alberti, Cs\"ornyei and Preiss \cite{MR2827846,MR2185733} and Alberti and Marchese \cite{MR3494485} on the converse to Rademacher's theorem in Euclidean space.

After \cite{Bate_2014}, Alberti representations have been used to study other questions in the setting of metric spaces \cite{David_2015,1611.05284,MR3474327,1712.07139,Bate_Li_2017, MR3554700,MR3818093}.
The goal of this note is to use these recent developments, and connections to geometric measure theory and harmonic analysis, to give another independent proof of Cheeger's theorem.
By considering Alberti representations, the main difference in our strategy to the strategies of Cheeger--Keith--Kleiner--Mackay is that we mostly study 1-rectifiable subsets of $X$, rather than Lipschitz functions defined on $X$.
That is, in some sense, we study the ``tangent bundle'' of $X$, rather than the ``co-tangent bundle''.

Let $X$ be a metric space and $f\colon X\to\R$ Lipschitz.
A Borel $\rho\colon X \to \R$ is an \emph{upper gradient} of $f$ if, for every 1-Lipschitz $\gamma\colon [0,1]\to X$,
\[|f(\gamma(1))-f(\gamma(0))| \leq \int_0^1 \rho(\gamma(t)) \d t.\]
Our main result is the following.
\begin{theorem}\label{main} Let $(X,d,\mu)$ be a metric measure space with $\Hd X <\infty$.
  Suppose that there exists a $C\geq 1$ such that, for every Lipschitz $f\colon X \to \R$ and every upper gradient $\rho$ of $f$,
  \begin{equation}\label{Lipug}\Lip(f,x) \leq C\rho(x) \quad \mu \text{-a.e.\ } x\in X.\end{equation}
  Then $(X,d,\mu)$ is a Lipschitz differentiability space. Further, one can take $n_{i} \leq \Hd X$ for all $i \in \N$ in \Cref{lipDiffSpace}. \end{theorem}
It is easy to check that \eqref{Lipug} holds in any PI space.
Indeed, it is obtained by combining the Poincar\'e inequality (see \cite[eq.\ 4.3]{cheeger-diff}) with \cite[Proposition 4.3.3]{keith}, and applying the Lebesgue differentiation theorem.
See also \cite[Proposition 4.26]{cheeger-diff}.
Consequently, \Cref{cheeger} follows from \Cref{main}.

For a Lipschitz $\phi\colon X \to \R^{n}$, the set
	\[\ind\phi = \{x\in X : \Lip(D\cdot \phi,x)>0 \ \forall D\in\mathbb S^{n-1}\}\]
will play a key role in this article. Indeed, after a short reduction (see \Cref{prop:chart-decomp}), to prove \Cref{main} it is sufficient to establish the following: if $n > \Hd X$, then $\mu(\ind\phi)=0$ for any Lipschitz $\phi\colon X \to \R^n$. The proof divides into the following two steps.

First, in \Cref{alberti-decomposition}, given a Lipschitz $\phi\colon X \to \R^n$, we find a decomposition $\ind\phi= \bigcup_{i = 1}^{\infty} A_{i} \cup S$ where $\mu \llcorner A_{i}$ has $n$ independent Alberti representations, and $S$ belongs to a certain class of singular sets denoted by $\mathcal S(\phi)$ (see \Cref{def:atilde}).
Under the hypothesis \eqref{Lipug}, $\mathcal S(\phi)$ subsets of $\ind\phi$ have measure zero (see \Cref{lem:c-null-in-PI}).
Therefore, if $\mu(\ind\phi)>0$, there exists $A\subset X$ of positive measure such that $\mu \llcorner A$ has $n$ independent Alberti representations.

The second step is to show that any positive measure on $X$ with $n$ independent Alberti representations -- in particular $\mu\llcorner A$ -- has Hausdorff dimension at least $n$. This implies that $n \leq \dim_{\mathrm{H}} A \leq \dim_{\mathrm{H}} X$ and completes the proof.

Recent developments give several ways to accomplish this second step.
First, the statement follows from a deep result of De Philippis and Rindler \cite{dephilippisrindler} regarding the structure of measures on Euclidean space satisfying a PDE constraint.
Indeed, using this result, De Philippis, Marchese and Rindler \cite{MR3701738} show that any measure on $\mathbb R^n$ with $n$-independent Alberti representations is absolutely continuous with respect to Lebesgue measure.
Similarly, the required dimension estimate follows from the related work of Arroyo-Rabasa \cite{MR4042849}.

Here we give another, elementary, proof (see \Cref{thm:alberti-dimension}) using a version of the multilinear Kakeya inequality. The inequality we need is due to Guth \cite[Theorem 7]{MR3300318}. It is a generalisation (for neighbourhoods of Lipschitz graphs) of the original multilinear Kakeya inequality (for nearly axis-parallel tubes), due to Bennett, Carbery, and Tao \cite[Theorem 1.15]{MR2275834}. 

We also mention that David \cite{David_2015} and Schioppa \cite{MR3474327} showed earlier that the \emph{Assouad dimension} of $A$ is at least $n$. This conclusion is weaker than ours, since Assouad dimension is generally larger than Hausdorff dimension.
Unlike these results, and also unlike the approaches of Cheeger--Keith--Kleiner--Mackay, by using the multilinear Kakeya inequality, our approach does not involve Gromov--Hausdorff tangents of $X$.

The paper consists of six sections.
In \Cref{sec:reduction} we prove the reduction involving $\ind\phi$ mentioned above.

In \Cref{sec:fragments} we give the definition of an Alberti representation and develop a simple, intrinsic construction of an Alberti representation.
In particular, our proof of \Cref{lem:c-null} is new.  Previous approaches relied on isometric embeddings into Banach spaces.

In \Cref{sec:alberti} we define independent Alberti representations and prove the decomposition result involving $\mathcal S(\phi)$ sets mentioned above.
The results in this section are a refinement of results from \cite{Bate_2014}.

In \Cref{sec:alberti-hausdorff} we use Guth's proof of the multilinear Kakeya inequality to show that any set which supports a positive measure with $n$ independent Alberti representations must have Hausdorff dimension at least $n$.

In \Cref{conclusion}, we give a general criterion for a space $(X,d,\mu)$ to be a Lipschitz differentiability space (see \Cref{thm:main}) and show that it implies a stronger version of \Cref{main} (see \Cref{main-star}).

Finally, we mention that both \Cref{thm:main} and \Cref{main-star} are characterisations of Lipschitz differentiability spaces (and hence also of spaces satisfying Keith's hypotheses, see \cite[Corollary 10.5]{Bate_2014}), up to a countable decomposition of $X$.
The converse directions are a direct consequence of \cite[Theorem 6.9]{Bate_2014} and \cite[Theorem 7.8]{Bate_2014} respectively, and were the initial motivation for studying Alberti representations of metric measure spaces.

\section{A reduction}\label{sec:reduction}
We require the following standard result of measure theory:
Suppose that $\mu$ is a $\sigma$-finite measure on $X$ and that $\mathcal T$ is a collection of $\mu$ measurable sets such that any positive measure subset of $X$ contains an element of $\mathcal T$ of positive measure. 
Then $\mu$ almost all of $X$ can be decomposed into a countable union of sets from $\mathcal T$.

\begin{proposition}\label{prop:chart-decomp}
        Let $(X,d,\mu)$ be a metric measure space and
        suppose that there exists $N\in\N$ such that, for any
        Lipschitz $\phi\colon X \to \R^N$, $\mu(\ind\phi)=0$.
        Then $(X,d,\mu)$ is a Lipschitz differentiability space. Moreover, in \Cref{lipDiffSpace}, we can take $n_{i} \leq N - 1$ for all $i \in \N$.
\end{proposition}

\begin{proof} We start with an observation: if $\phi \colon X \to \R^{n}$, $x_{0} \in \ind \phi$, and $f$ is \textbf{not} differentiable with respect to $\phi$ at $x_{0}$, then $x_{0} \in \ind (\phi,f)$. To see this, assume to the contrary that $x_{0} \notin \ind (\phi,f)$. Then there exists $D = (d_{1},\ldots,d_{n + 1}) \in \mathbb{S}^{n}$ such that
\begin{equation}\label{form2} \limsup_{x \to x_{0}} \frac{|D \cdot ((\phi,f)(x) - (\phi,f)(x_{0}))|}{d(x,x_{0})} = 0. \end{equation} 
Since $x_{0} \in \ind \phi$, we have $d_{n + 1} \neq 0$, and hence we can define $D_{1} := (\tfrac{d_{1}}{d_{n + 1}},\ldots,\tfrac{d_{n}}{d_{n + 1}}) \in \R^{n}$. It now follows from \eqref{form2} that
\begin{equation}\label{form3} \limsup_{x \to x_{0}} \frac{|(f(x) - f(x_{0})) - D_{1} \cdot (\phi(x) - \phi(x_{0}))|}{d(x,x_{0})} = 0. \end{equation} 
This implies that $f$ is differentiable at $x_{0}$, contrary to our assumption. A fine point to observe here is that the vector $D_{1} \in \R^{n}$ satisfying \cref{form3} is unique, which is part of the requirement that $f$ is differentiable at $x_{0}$: if there existed another $D_{2} \in \R^{n}$ such that that \eqref{form3} holds, then $\Lip((D_{2} - D_{1}) \cdot \phi,x_{0}) = 0$, which implies $D_{1} = D_{2}$ by the assumption $x_{0} \in \ind \phi$.

	  Now we start the proof of the proposition in earnest.  Let $\mathcal T$ be the collection of all Borel $U\subset X$ for which there exist an $n\in \N$ and a Lipschitz $\phi\colon X \to \R^n$ such that every Lipschitz $f\colon X \to \R$ is differentiable $\mu$-a.e.\ in $U$ with respect to $\phi$.

        Let $U \subset X$ be a Borel set of positive measure.
        Either every Lipschitz $\phi \colon X \to \R$ satisfies $\Lip(\phi,x)=0$ for $\mu$ almost every $x\in U$ or there exists a Lipschitz $\phi_1\colon X \to \R$ with $\mu(\ind\phi_1\cap U)>0$.
        Given the first option we stop; in this case $U\in \mathcal T$, because every Lipschitz $f \colon X \to \R$ is differentiable with respect to the constant map $\phi \colon X \to \R^{0} = \{0\}$. 
        
        Otherwise we proceed iteratively. Suppose that, for some $n\in \N$, there exists a Lipschitz $\phi_n \colon X \to \R^n$ with $\mu(\ind\phi_n\cap U)>0$.
        Then either
        \begin{itemize}
                \item every Lipschitz $f\colon X \to \R$ is differentiable $\mu$-a.e.\ in $\ind\phi_n\cap U$ with respect to $\phi_n$, or
                \item there exists $\phi^{n + 1} \colon X \to \R$ which is not differentiable with respect to $\phi_{n}$ in a subset of $\ind \phi_{n}$ of positive $\mu$ measure. In this case, by the observation at the beginning of the proof, $\phi_{n+1}:=(\phi_{n},\phi^{n+1})$ satisfies $\mu(\ind\phi_{n+1}\cap U)>0$.
        \end{itemize}
        Note that, in the first option, the uniqueness of the derivative of any Lipschitz $f \colon X \to \R^{n}$ with respect to $\phi_{n}$ at all $x_{0} \in \ind \phi_{n} \cap U$ is guaranteed because $x_{0}\in \ind\phi_{n}$. As we already saw before, if $D_{1},D_{2}$ are two derivatives of a $f$ at $x_{0}$, then $\Lip((D_{1}-D_{2})\cdot \phi_{n},x_{0})=0$, so that $D_{1}=D_{2}$.

        By hypothesis, the second option cannot hold for $n = N-1$, and so for some $n \leq N - 1$ the first option holds. Then $\ind\phi_n \cap U \in\mathcal T$.
	    Applying the standard measure theory result to $\mathcal T$ completes the proof.
\end{proof}

\section{Alberti representations and analysis of curve fragments}\label{sec:fragments}

We will consider the following structure of 1-rectifiable subsets of a metric measure space.
Let $I=[0,1]$.
Given a metric space $X$ we write $\mathcal K(X)$ for the set of non-empty compact subsets of $I\times X$ equipped with the Hausdorff metric $d_{H}$.
If $X$ is complete and separable, respectively compact, then so is $\mathcal K(X)$.
\begin{definition}
	\label{def:Gamma}
	For a metric space $X$, the set of \emph{curve fragments} in $X$ is
        \[\Gamma(X)=\{\gamma\colon \dom\gamma \subset I \to X : \dom\gamma \text{ compact, non-empty, }\gamma \text{ 2-biLipschitz}\}.\]
        We emphasise that the domain of each $\gamma\in \Gamma(X)$ does not need to be connected.

        Given $\gamma\in \Gamma(X)$, the \emph{graph} of $\gamma$ is
        \[\graph(\gamma) = \{(t,\gamma(t)) : t\in\dom\gamma\} \subset I \times X.\]
        We metrise $\Gamma(X)$ by setting $d_{\Gamma}(\gamma_{1},\gamma_{2}) := d_{H}(\graph(\gamma_{1}),\graph(\gamma_{2}))$. This formula indeed defines a metric on $\Gamma(X)$, because if $d_{H}(\graph(\gamma_{1}),\graph(\gamma_{2})) = 0$, then $\dom \gamma_{1} = \dom \gamma_{2}$ and $\gamma_{1} \equiv \gamma_{2}$ on $\dom\gamma_{1}$.
        \end{definition}
Whenever it is clear from the context, we will identify $\gamma$ with its image $\im\gamma = \gamma(\dom \gamma)$. We will consider integral combinations of the measures 
\begin{displaymath} \gamma_{*}\L^{1} := \gamma_{\ast}\mathcal{L}^{1}\llcorner (\dom \gamma), \end{displaymath}
where, in general, $f_{\ast}\nu$ refers to the push-forward of a measure $\nu$ under a map $f$.
To do so, we first consider the measurability of the integrand.
\begin{lemma}\label{pushforward-is-continuous}
 Let $X$ be a metric space and $C\subset X$ closed.
 Then
 \[\Gamma(X) \ni \gamma \mapsto \gamma_* \L^{1} (C)\]
 is upper semi-continuous.
 \end{lemma}

\begin{proof}
	Let $\gamma_n \to \gamma$ in $\Gamma(X)$.
	Observe that, for every $\epsilon>0$,
	\[\gamma_n^{-1}(C) \subset  B(\gamma^{-1}(C),\epsilon)\]
	for sufficiently large $n$.
	Indeed, if this were not true then (after passing to a subsequence) for each $n\in\N$ there exists $t_n\in\gamma_n^{-1}(C)$ with $d(t_n,\gamma^{-1}(C))\geq \epsilon$.
	There exists $t\in I$ such that (after passing to a further subsequence)  $t_{n}\to t$ and so, since  $\gamma_n\to \gamma$, $\gamma_n(t_n)\to \gamma(t)$.
	Since $C$ is closed, $\gamma(t)\in C$, hence $t \in \gamma^{-1}(C)$, and this contradicts $d(t_{n},\gamma^{-1}(C)) \geq \epsilon$ for all $n\in\N$.

        It follows that
        \[\epsilon_{n} := \L^{1}(\gamma_n^{-1}(C)\setminus \gamma^{-1}(C)) \leq \L^{1}(B(\gamma^{-1}(C),\epsilon)\setminus \gamma^{-1}(C))\]
	for $n \in \N$ sufficiently large. Since $C \subset X$ is closed, the right hand side above can be taken arbitrarily small by choosing $\epsilon > 0$ small. This implies that
	\begin{displaymath} \limsup_{n \to \infty} \, (\gamma_{n})_{\ast}\mathcal{L}^{1}(C) \leq \gamma_{\ast}\mathcal{L}^{1}(C) + \limsup_{n \to \infty} \, \epsilon_{n} = \gamma_{\ast}\mathcal{L}^{1}(C), \end{displaymath}
	as required.
	\end{proof}

\begin{lemma}\label{lem:integrand-is-measurable}
  Let $X$ be a metric space and $\P$ a finite Borel measure on $\Gamma(X)$.
  For any Borel $B\subset X$, the map $\gamma\mapsto \gamma_{*}\L^{1}(B)$ is Borel, and the set function $\nu(\P)$ defined by
  \begin{equation}
    \label{eq:defn-of-measure}
    \nu(\P)(B)=\int_{\Gamma(X)} \gamma_{\ast}\mathcal{L}^{1}(B) \d \P(\gamma)
  \end{equation}
 for any Borel $B\subset X$ is a Borel measure.
\end{lemma}
\begin{proof}

It is easy to check that the set family 
\begin{displaymath} \mathcal{A} := \{A \subset X \text{ Borel} : \gamma \mapsto \gamma_{\ast}\mathcal{L}^{1}(A) \text{ is a Borel function}\} \end{displaymath} 
is a \emph{Dynkin system}: $X \in \mathcal{A}$, and $\mathcal{A}$ is stable under taking complements and countable unions of disjoint sets. For example: if $A \in \mathcal{A}$, then $\gamma \mapsto \gamma_{\ast}\mathcal{L}^{1}(A^{c}) = \gamma_{\ast}\mathcal{L}^{1}(X) - \gamma_{\ast}\mathcal{L}^{1}(A)$ is a Borel function, so $A^{c} \in \mathcal{A}$. Moreover, \Cref{pushforward-is-continuous} states that $\mathcal{A}$ contains the ($\pi$-system of) closed sets, so it follows from Dynkin's lemma that $\mathcal{A}$ equals the Borel sets. Hence the formula \eqref{eq:defn-of-measure} makes sense for all Borel $B \subset X$.
Once well defined, it is clear that $\nu(\P)$ defines a measure.
\end{proof}


 

\begin{definition}
  \label{def:alberti-representation}
          An \emph{Alberti representation} of a metric measure space $(X,d,\mu)$ consists of a Borel probability measure $\P$ on $\Gamma(X)$ such that $\mu \ll \nu(\P)$.
\end{definition}

The motivation for considering Alberti representations in \cite{Bate_2014} was the following.
Suppose that $f\colon X \to \R$ is Lipschitz and that $\gamma\in \Gamma(X)$.
Then the composition
\[f\circ \gamma \colon \dom\gamma \to \R\]
is Lipschitz and so is differentiable almost everywhere.
That is, for $\H^1$-a.e.\ $x \in \im\gamma$, there exists a \emph{partial derivative} of $f$ at $x$ given by $(f\circ \gamma)'(\gamma^{-1}(x))$.
The existence of an Alberti representation is precisely the condition required for such a partial derivative to exist $\mu$ almost everywhere on $X$.

In this article, we will only be interested in the geometric structure an Alberti representation imposes.
To describe this structure, we make the following definition. 

\begin{definition}\label{def:curve-direction}
        Let $X$ be a metric space, $\phi\colon X \to \R^n$ Lipschitz and $C \subset \R^n$ a cone. By a \emph{cone} in $\R^{n}$, we refer to sets of the form $C = \{v \in \R^{n} : |v \cdot w| \geq \theta\|v\|\}$ for some $w \in \S$ and $\theta \in (0,1)$.
 We say that $\gamma\in \Gamma(X)$ is a \emph{$C$-curve (with respect to $\phi$)} if there exists a $\delta>0$ such that
        \[\phi(\gamma(t))-\phi(\gamma(s))\in C\setminus B(0,\delta |s-t|) \quad \forall s,t \in\dom\gamma.\]

 An Alberti representation $\P,$ is a \emph{$C$-Alberti representation (with respect to $\phi$)} if $\P$-a.e.\ $\gamma$ is a $C$-curve (with respect to $\phi$).
                A Borel $S\subset X$ is \emph{$C$-null} (with respect to $\phi$) if $\H^1(\gamma\cap S)=0$ for each $C$-curve $\gamma$ (with respect to $\phi$).

        Often the function $\phi$ will be clear from the context and we will not explicitly mention it.
        In the case that $X=\R^n$, we will only consider $\phi=\operatorname{id}$.
\end{definition}

Observe that if $(X,d,\mu)$ supports a $C$-Alberti representation, then any $C$-null subset of $X$ is $\mu$-null.
In this section we show that the converse statement holds.

For a metric space $X$, let $\mathcal{M}(X)$ and $\mathcal{P}(X)$ be the set of Borel measures on $X$ with total mass $\leq 1$ and $= 1$, respectively. Note that $\mathcal{M}(X)$ and $\mathcal{P}(X)$ are compact metric spaces with respect to weak* convergence if $X$ is compact.
\begin{lemma}
  \label{lem:alberti-measures-are-compact}
  Let $X$ be a metric space and $K\subset \Gamma(X)$ compact.
  Then the set
  \[\mathcal N(K)=\{m\in\M(X): \exists \, \P\in \mathcal{P}(K) \text{ such that } m\leq \nu(\P)\}\]
  is convex and compact.
\end{lemma}

\begin{proof}
  Evidently $\mathcal N(K)$ is convex. To show compactness, first note that every measure in $\mathcal{N}(K)$ is supported on the compact set
  \begin{displaymath} \im K := \cup \{\im \gamma : \gamma \in K\}. \end{displaymath}
  Hence, replacing $X$ by $\im K$, we may assume that $X$ is compact. Now, let $\{m_{n}\}_{n \in \N} \subset \mathcal{N}(K) \subset \mathcal{M}(X)$. After passing to a subsequence, we may assume that $m_{n} \rightharpoonup m \in \mathcal{M}(X)$. So, it remains to show that $m \in \mathcal{N}(K)$. To this end, let $\P_{n}\in \mathcal{P}(K)$ be such that $m_{n}\leq \nu(\P_{n})$.
  Since $K$ is compact, we may suppose that there exists $\P\in \mathcal{P}(K)$ with $\P_{n}\rightharpoonup \P$. We then argue that $m \leq \nu(\tn)$.

  Let $C \subset X$ be closed, $\delta > 0$, and $C(\delta) := \{x \in X : \dist(x,C) < \delta\}$. By \Cref{pushforward-is-continuous}, the sets $\{\gamma \in K : \gamma_{\ast}\mathcal{L}^{1}(\overline{C(\delta)}) \geq \lambda\}$ are compact for $\lambda \in \R$.
  Therefore, using that $m_{n}\rightharpoonup m$, $\P_{n}\rightharpoonup \P$, $\gamma_{\ast}\mathcal{L}^{1}(X) \leq 1$, and the reverse Fatou lemma,
  \begin{align*}
    m(C) \leq \liminf_{n\to\infty}m_{n}(C(\delta)) & \leq \limsup_{n\to\infty} \int_{K} \gamma_{*}\L^{1}(\overline{C(\delta)})\d\P_{n}(\gamma)\\
    & = \limsup_{n \to \infty} \int_{0}^{1} \P_{n}\{\gamma \in K : \gamma_{\ast}\mathcal{L}^{1}(\overline{C(\delta)}) \geq \lambda\} \, d\lambda\\
    & \leq \int_{0}^{1} \P\{\gamma \in K : \gamma_{\ast}\mathcal{L}^{1}(\overline{C(\delta)}) \geq \lambda\} \, d\lambda = \nu(\P)(\overline{C(\delta)}).  \end{align*}
The right hand side converges to $\nu(\tn)(C)$ as $\delta \to 0$, so we have the inequality $m(C) \leq \nu(\tn)(C)$ for closed sets $C \subset X$. For a general Borel set $B \subset X$ and any $\epsilon > 0$, there exists a closed set $C \subset B$ with $m(B \setminus C) < \epsilon$. This implies $m(B) \leq \nu(\tn)(B)$. \end{proof}

\begin{proposition}
 \label{thm:formal-construct-alberti}
 Let $X$ be a metric space, $\mu$ a finite Borel measure on $X$ and $K\subset \Gamma(X)$ compact.
 Then, there exists a Borel decomposition $X=A\cup S$ and a $\mathbb P \in\mathcal P(K)$ such that $\mu\llcorner A \ll \nu(\P)$ and
 $\H^1(\gamma\cap S)=0$ for every $\gamma\in K$. 
\end{proposition}

\begin{proof} Let $\mathcal N = \mathcal{N}(K)$ be as in \Cref{lem:alberti-measures-are-compact} and apply the GKS decomposition theorem \cite[Theorem 9.4.4]{rudin-unitball} to $\mathcal N$.
	The set $\mathcal N$ satisfies the hypotheses of the GKS decomposition theorem, except that its elements have total measure \emph{at most} 1, rather than exactly 1.
	However, the proof of the GKS theorem holds with minor superficial changes if we only assume the measures have total measure at most 1.
	
	The conclusion of the GKS theorem gives a decomposition $\mu = \mu_{1} + \mu_{2}$, a Borel set $S \subset X$, and a measure $m_{0} \in \mathcal{N}(K)$ such that $\mu_{1} \ll m_{0}$, $\mu_{2}$ is concentrated on $S$, and $m(S) = 0$ for all $m \in \mathcal{N}(K)$. In particular $\mu\llcorner A = (\mu_{1})\llcorner A \ll m_{0}$ for $A := X \setminus S$. Since $m_{0} \in \mathcal{N}(K)$, there further exists $\P \in \mathcal{P}(K)$ such that $\mu\llcorner A \ll \nu(\P)$, as claimed. Finally, if $\gamma\in K$, note that $\nu(\delta_{\gamma}) \in\mathcal N$, so $\gamma_{\ast}\mathcal{L}^{1}(S) = \nu(\delta_{\gamma})(S) = 0$. Since $\gamma \colon \dom \gamma \to X$ is biLipschitz, this implies that $\mathcal{H}^{1}(\gamma \cap S) = 0$. \end{proof}

By suitable measure theoretic manipulations, we arrive to the following characterisation of the existence of Alberti representations.

\begin{theorem}\label{lem:c-null}
        Let $(X,d,\mu)$ be a metric measure space, $\phi\colon X \to \R^n$ Lipschitz and $C\subset \R^n$ a cone.
        There exists a decomposition $X= A \cup S$ such that $\mu\llcorner A$ has a $C$-Alberti representation and $S$ is $C$-null.
\end{theorem}

\begin{proof}
    Let $\Gamma_{C}$ be the set of $C$-curves in $\Gamma(X)$. Since $(X,d)$ is complete and separable and $\mu$ is $\sigma$-finite, there exists an increasing sequence of compact sets $X_j\subset X$ of finite measure such that $\mu(X\setminus \cup_j X_j)=0$.
	Fix a $j\in\N$.
	We will identify $\Gamma(X_j)$ with its isometric copy in $\Gamma(X)$. Note that the set $\Gamma_{C} \cap\Gamma(X_{j})$ is not compact, but it is $\sigma$-compact.
    Indeed, for each $i\in\N$, the set $K_i$ of $\gamma \in \Gamma_{C} \cap \Gamma(X_{j})$ with
    \[\|\phi(\gamma(t))-\phi(\gamma(s))\| \geq |t-s|/i\]
    is compact and $\cup_i K_i = \Gamma_{C} \cap \Gamma(X_{j})$.
    
    Apply \Cref{thm:formal-construct-alberti} to $\mu \llcorner X_{j}$ and $K_i$.
    This gives a measure $\P_i \in \mathcal{P}(K_i)$ and Borel decomposition $X_j=A_i \cup S_i$ such that $\mu\llcorner A_i \ll \nu(\P_i)$ and $\H^1(\gamma\cap S_i)=0$ for every $\gamma\in K_i$.
    Repeat this for each $i\in\N$ and set
    \[\P^j = \sum_i 2^{-i}\P_i \in \mathcal{P}(\Gamma_{C} \cap \Gamma(X_{j})), \quad A^j=\cup_i A_i, \quad S^j=\cap_i S_i.\]
    Then $X_j = A^j\cup S^j$ and $\H^1(\gamma\cap S^j)=0$ for every $\gamma\in \Gamma_{C} \cap\Gamma(X_j)$.
    Also, since $\nu(\P^j) = \sum_i 2^{-i}\nu(\P_i)$, we have $\mu\llcorner A^j \ll \nu(\P^j)$.

    Now set
    \[A=(X\setminus \cup_j X_j) \cup \bigcap_n \bigcup_{j>n} A^j, \quad S= \bigcup_n \bigcap_{j>n} S^j, \quad \P= \sum_{j} 2^{-j} \P^j \in \mathcal{P}(\Gamma_{C}).\]
    Then $\mu\llcorner A \ll \nu(\P)$.
    Also, since the $X_j$ are increasing, $X=A\cup S$ and $\H^1(\gamma\cap S)=0$ for every $\gamma \in \cup_j\Gamma(X_j) \cap \Gamma_{C}$.
    Finally, we can write any $\gamma\in \Gamma_{C}$ as
    \[\gamma = [(X\setminus \cup_j X_j) \cap \gamma] \cup \bigcup_{j} [X_j \cap \gamma],\]
    where each $X_j\cap \gamma \in \Gamma_{C} \cap \Gamma(X_j)$.
    Since $S\subset \cup_j X_j$, this implies that $S$ is $C$-null.
\end{proof}

\begin{remark}
	\Cref{lem:c-null} is often vacuous outside $\ind\phi$.
	Consider, for example, $\phi(x)=(x,x)$ on $\R$ and any cone centred on $(1,-1)$. 
\end{remark}

\section{Existence of independent Alberti representations}\label{sec:alberti}

In this section we give criteria for when a metric measure space possesses many ``independent'' Alberti representations.
We say that cones $C_1,\ldots,C_n \subset \R^n$ are \emph{independent} if $\{v_{1},\ldots,v_{n}\}$ is a linearly independent set for any choices of $v_{i} \in C_i\setminus\{0\}$, $1\leq i\leq n$.
\begin{definition}\label{def:independent}
		A collection of Alberti representations $\P_1,\ldots,\P_n$ of a metric measure space $(X,d,\mu)$
        is \emph{independent} if there exists a Lipschitz $\phi\colon X \to \R^n$ and independent cones $C_1,\ldots,C_n \subset \R^n$ such that $\P_i$ is a $C_i$-Alberti representation with respect to $\phi$ for each $1\leq i\leq n$.
\end{definition}

By applying \Cref{lem:c-null} multiple times using a collection of independent cones $C_1,C_2,\ldots,C_n$, we obtain a decomposition $X= A \cup S_1 \cup S_2 \cup \ldots \cup S_n$ where each $S_i$ is $C_i$-null and $\mu\llcorner A$ has $n$ independent Alberti representations.
However, in practice, we can only deduce that a $C$-null set is $\mu$-null whenever $C$ is very wide.
We now work towards a similar decomposition involving wide cones.

\begin{lemma}\label{lem:cone-null}
        Let $X$ be a metric space and $\phi\colon X \to \R^n$ Lipschitz.
Suppose that $S\subset X$ is $C$-null with respect to $\phi$.
Then for any $\gamma \in \Gamma(X)$,
\[\mathcal{H}^{1}(\{t \in \gamma^{-1}(S) : (\phi\circ\gamma)'(t) \in \inti(C)\}) = 0.\]
\end{lemma}

\begin{proof}
Suppose that the conclusion is false: for some $\gamma\in \Gamma(X)$
\[\mathcal{H}^{1}(\{t\in \gamma^{-1}(S) : (\phi\circ\gamma)'(t) \in \inti(C)\}) > 0.\]
Then, we can pick $i \in \N$ and $R > 0$ such that
\[B_{i,R} = \Big\{t\in \gamma^{-1}(S) : \frac{\phi(\gamma(t+r))-\phi(\gamma(t))}{r} \in C\setminus B(0,\tfrac{1}{i}) \ \forall\ r \in \dom\gamma \cap (0,R)\Big\}\]
has positive measure. Further, we can find a compact subset $K \subset B_{i,R}$ of positive measure with $\mathrm{diam}(K) < R$.  In particular, for any $s,t\in K \subset \gamma^{-1}(S)$,
\[\phi(\gamma(t))-\phi(\gamma(s)) \in C\setminus B(0,\tfrac{1}{i} |t-s|).\]
Then $\gamma\llcorner K$ is a $C$-curve intersecting $S$ in a set of positive measure, and a contradiction has been reached. \end{proof}

The previous lemma gives us a method to ``refine" the cones associated to an Alberti representation. One should think of the cones $C_i$ in the following lemma as being very thin.
\begin{lemma}\label{refine}
Let $X$ be a metric space, $\phi\colon X \to \mathbb R^n$ Lipschitz and $C\subset \mathbb R^n$ a cone.
Suppose that a measure $\mu$ has a $C$-Alberti representation.
Then, for any collection of cones $C_1,\ldots,C_m \subset \mathbb R^n$ such that
\[C \setminus \{0\} \subset \bigcup_{i=1}^m \inti(C_{i}),\]
there exists a Borel decomposition $X= A_1\cup\ldots\cup A_m$ such that each $\mu\llcorner A_i$ has a $C_i$-Alberti representation.
\end{lemma}

\begin{proof}
By \Cref{lem:c-null} there exists a decomposition $X=A_1\cup S_1$ such that $A_1$ has the required form and $S_1$ is $C_1$-null.
By applying \Cref{lem:c-null} again (to $\mu \llcorner S_{1}$) we obtain a decomposition $X=A_1\cup A_2 \cup S_2$ where $A_2$ has the required form and $S_2$ is both $C_1$-null and $C_2$-null. Repeating, we obtain a decomposition $X=A_1\cup\ldots\cup A_m \cup S$ such that each $A_i$ has the required form and $S$ is $C_i$-null for each $1\leq i \leq m$. Now, if $\gamma$ is a $C$-curve, then 
\begin{align*} \mathcal{H}^{1}(\gamma^{-1}(S)) & = \mathcal{H}^{1}(\{t \in \gamma^{-1}(S) : (\phi \circ \gamma)'(t) \in C \setminus \{0\}\})\\
& \leq \sum_{i = 1}^{m} \mathcal{H}^{1}(\{t \in \gamma^{-1}(S) : (\phi \circ \gamma)'(t) \in \inti(C_{i})\}) = 0 \end{align*}
by \Cref{lem:cone-null}, which implies that $S$ is $C$-null. Since $\mu$ has a $C$-Alberti representation, we have $\mu(S)=0$.
Therefore, $A'_1=A_1\cup S$ also has the required properties of $A_1$, and $X=A'_1\cup A_2\cup \ldots \cup A_m$ is the required decomposition.
\end{proof}

\begin{definition}\label{def:atilde}
For $w\in \S$ and $0<\theta<1$, recall the notation $C(w,\theta) = \{v\in\R^n: |v\cdot w| \geq \theta\|v\|\}$. Note that $C(w,\theta)$ becomes wider as $\theta\to 0$. Let $X$ be a metric space and $\phi\colon X \to \R^n$ Lipschitz. Define the set $\mathcal{S}(\phi)$ to be the collection of Borel $S\subset X$ for which the following is true:
For any $0<\theta<1$ there exists a Borel decomposition
\[S = S_1\cup S_2 \cup \ldots \cup S_m\]
and $w_1,\ldots,w_m \in \mathbb S^{n-1}$ such that $S_i$ is $C(w_i,\theta)$-null for each $1\leq i \leq m$.
\end{definition}

\begin{theorem}\label{alberti-decomposition}
        Let $(X,d,\mu)$ be a metric measure space and $\phi\colon X \to \R^n$ Lipschitz.
        There exists a Borel decomposition
        \[X = S \cup \bigcup_{i\in\N} A_i\]
        such that each $\mu\llcorner A_i$ has $n$ independent Alberti representations and $S\in \mathcal{S}(\phi)$.
\end{theorem}

\begin{proof}
We first prove the following.
Given any $0<\theta<1$ and $1\leq d \leq n$ there exists a Borel decomposition
\begin{equation}\label{eq:alberti-dec}
X=\bigcup_{i=1}^m A_i \cup \bigcup_{i=1}^m S_i
\end{equation}
such that each $\mu\llcorner A_i$ has $d$ independent Alberti representations and each $S_i$ is $C(w_i,\theta)$-null for some $w_{i} \in \S$. We find these representations iteratively as follows.

        First apply \Cref{lem:c-null} to an arbitrary cone $C = C(w,\theta)$ to obtain a decomposition $X=A\cup S$ where $\mu\llcorner A$ has a $C$-Alberti representation and $S$ is $C$-null.

		Choose $\alpha = \alpha(\theta) < 1$ such that if $w_1,\ldots,w_{n} \in \S$ and $w_n \perp w_i$ for $1 \leq i<n$, then 
		\begin{equation}\label{form1} C(w_n,\theta) \cap C(w_i,\alpha) = \{0\}, \qquad 1 \leq i<n. \end{equation}
        Suppose then that, for some $1\leq d <n$, there exists a decomposition as in \cref{eq:alberti-dec} where each $\mu\llcorner A_i$ has $d$ independent Alberti representations.
	    By applying \Cref{refine} (and increasing $m$) we may suppose that there exists $w_1^i,\ldots,w_d^i\in\S$ such that these representations, for $i$ fixed, are $C(w_j^i,\alpha)$-Alberti representations for each $1 \leq j \leq d$, and that these cones are independent.
        Fix $1\leq i \leq m$ and pick $w_{d+1}^i \in \S \cap \mathrm{span}(w_1^i,\ldots, w_d^i)^{\perp}$.
        Then the cones 
        \begin{displaymath} C(w_1^i,\alpha),\ldots, C(w_d^i,\alpha) \quad \text{ and } \quad C(w_{d+1}^i,\theta) \end{displaymath}
        are independent by \cref{form1}.
        Applying \Cref{lem:c-null} once more gives a decomposition $A_i = A'_i \cup S'$ where $\mu\llcorner A'_i$ has $d+1$ independent Alberti representations and $S'$ is $C(w_{d+1}^i,\theta)$-null.

        Repeating this for each $1\leq d <n$ completes the proof of \cref{eq:alberti-dec}.
        To complete the proof of the \Cref{alberti-decomposition}, we apply \cref{eq:alberti-dec} for each $\theta= 1/j$, $j \in \N$, to obtain a decomposition $X=\hat A_j \cup \hat S_j$ where each $\hat A_j$ is a finite union of sets with $n$ independent Alberti representations, and each $\hat S_j$ has a decomposition $\hat{S}_j=S_j^1\cup\ldots\cup S_{j}^{m_j}$, where each $S_j^i$ is $C(w_j^i,\theta)$-null  for some $w_j^i\in\S$.
        Setting $S=\cap_j \hat S_j$ completes the proof.
\end{proof}

\section{Alberti representations and Hausdorff dimension}\label{sec:alberti-hausdorff}
In this section we show, using Guth's multilinear Kakeya inequality \cite[Theorem 7]{MR3300318} for neighbourhoods of Lipschitz graphs, that any non-trivial measure with $n$ independent Alberti representations has lower Hausdorff dimension at least $n$. Recall that the lower Hausdorff dimension of a Borel measure $\mu$ is 
\begin{displaymath} \Hd \mu := \inf \{\Hd B : B \text{ is Borel and } \mu(B) > 0\}. \end{displaymath}

Fix $n\geq 2$ and for $m\in\N$ let $\calD_{m}$ be the family of dyadic sub-cubes of $[0,1)^{n}$ of side-length $2^{-m}$.
For $Q\in \calD_m$ write $|Q|=2^{-mn}$ for the Lebesgue measure of $Q$.
For $0<\delta <1/4$ and for each $1 \leq j \leq n$ let $\Gamma_{j}(\delta) \subset \Gamma(\R^{n})$ be the set of $C(e_j,1 - \delta)$-curves in $\R^n$ (with respect to $\phi = \mathrm{id}$) and let $\tn_j \in \mathcal P(\Gamma_j(\delta))$. Recall that that the cone $C(e_{j},1 - \delta)$ gets narrower as $\delta \to 0$.
 
 For each $1\leq j\leq n$, cover $\Gamma_j(\delta)$ by a countable collection $\mathcal{R}_j$ of disjoint non-empty subsets of diameter at most $2^{-m}$. For each $1\leq j\leq n$ and every $R \in \mathcal{R}_j$, pick one curve fragment $\gamma \in R$, extend it to a $C(e_i,1 - \delta)$-curve $\gamma_R$ homeomorphic to $\R$, and let
 \begin{align*} f_{j}  & := 2^{m(n-1)}\sum_{R \in \mathcal{R}_j} \tn_{j}(R)\chi_{B(\im \gamma_{R},2\sqrt{n}2^{-m})}. \end{align*}
 
 We will use the notation $\lesssim_{\epsilon}$ for an inequality that holds up to a constant that depends on $n$ and $\epsilon$.
 The following is a direct consequence \cite[Theorem 7]{MR3300318}.
 \begin{theorem}\label{thm:guth}
 	For every $\epsilon>0$ there exists a $\delta>0$ such that
 \begin{displaymath} \int_{[0,1]^{n}} \prod_{j = 1}^{n} [f_{j}(x)]^{\frac{1}{n - 1}} \, dx \lesssim_{\epsilon} 2^{m\epsilon}. \end{displaymath}
 \end{theorem}
 \begin{proof}
  Guth formulates his result in terms of unit neighbourhoods of Lipschitz graphs, so to make his result applicable we first need to precompose each $f_{j}$ with a dilation by $\lambda_m:= 2^m/2\sqrt{n}$:
 \begin{align}\label{form14} \int_{[0,1]^{n}} \prod_{j = 1}^{n} [f_{j}(x)]^{\frac{1}{n - 1}} \, dx &= \lambda_m^{-n} \int_{[0,\lambda_m]^{n}} \prod_{j = 1}^{n} \left[ f_{j}\left(\frac{x}{\lambda_m} \right) \right]^{\frac{1}{n - 1}} \, dx\notag\\
 &= \int_{[0,\lambda_m]^{n}} \prod_{j = 1}^{n} \left[ f^{m}_{j}(x) \right]^{\frac{1}{n - 1}} \, dx.\end{align} 
 Here 
 \begin{displaymath} f^{m}_{j}(x) = \lambda_m^{1-n} \cdot f_{j}\left(\frac{x}{\lambda_m} \right) = C_n\sum_{R \in \mathcal{R}_j} \tn_{j}(R)\chi_{B(\im \lambda_m\gamma_{R},1)}(x), \end{displaymath} 
 where $\lambda_m\gamma_{R}$ is the $\lambda_m$-dilation of $\gamma_{R}$ (which is still a $C(e_j,1-\delta)$-curve), and $C_n$ is a constant depending on $n$. Thus, applying the weighted version of \cite[Theorem 7]{MR3300318} (as in \cite[Corollary 5]{MR3300318}) to the right hand side of \eqref{form14} gives
 \begin{displaymath} \int_{[0,1]^{n}} \prod_{j = 1}^{n} [f_{j}(x)]^{\frac{1}{n - 1}} \, dx \lesssim_{\epsilon} 2^{m\epsilon} \prod_{j = 1}^{n} \left(\sum_{R \in \mathcal{R}_j} \tn_{j}(R) \right)^{\frac{1}{n - 1}} \leq 2^{m\epsilon}, \end{displaymath}
 recalling that the sets $R \in \mathcal{R}_j$ are disjoint for each $1\leq j \leq n$.
 \end{proof}
 
 \begin{lemma}\label{prop:apply-guth}
 	For every $\epsilon>0$ there exists a $\delta>0$ such that the following is true.
 	Suppose that $\mu$ is a measure on $\R^{n}$ for which there exist $M_1,\ldots,M_n>0$ such that $\mu \leq M_{j}\nu(\tn_{j})$ for all $1 \leq j \leq n$, where $\P_{j} \in \mathcal{P}(\Gamma_{j}(\delta))$. Then
\begin{equation}\label{mainMult} \sum_{Q \in \calD_{m}} \left(\frac{\mu(Q)}{|Q|} \right)^{\frac{n}{n - 1}}|Q| \lesssim_{\epsilon} 2^{m\epsilon} \prod_{j = 1}^{n}M_j^{\frac{1}{n - 1}}, \qquad m \geq 0. \end{equation}
\end{lemma}

\begin{proof}
Fix $1\leq j\leq n$, $m \geq 0$, and $Q\in \calD_m$.
For any $\gamma\in R \in \mathcal R_j$ such that $\gamma\cap Q\neq \emptyset$,
 \begin{equation*}
 	Q \subset B(\im \gamma, \sqrt{n} 2^{-m}) \subset B(\im \gamma_R, 2\sqrt{n}2^{-m}).
 \end{equation*}
 Consequently, for any $x\in Q$,
 \begin{align*}\label{form44} f_{j}(x) & \geq 2^{m(n-1)} \sum \{\tn_{j}(R) : R \in \mathcal R_j, \exists \, \gamma \in R \text{ with } \im \gamma \cap Q \neq \emptyset\}\notag\\
 &\geq |Q|^{-\frac{n-1}{n}} \cdot \tn_{j}(\{\gamma \in \Gamma_{j} : \im \gamma \cap Q \neq \emptyset\}). \end{align*}
Further, $\gamma_{*}\L^{1}(Q) \lesssim |Q|^{\frac{1}{n}}$ for all $\gamma \in \Gamma_{j}(\delta) \subset \Gamma(\R^{n})$, which consists of $2$-biLipschitz curves, and so, by our assumption on $\mu$,
\begin{equation*}\label{form5}  \mu(Q) \lesssim  M_j \cdot |Q|^{\frac{1}{n}}\cdot \tn_{j}(\{\gamma\in \Gamma_i : \gamma \cap Q \neq \emptyset\}) \leq M_j |Q| f_j(x). \end{equation*}
Dividing by $|Q|$, taking a geometric average over $1\leq j\leq n$, and integrating the result over $Q$ gives
\begin{align*}\label{form11} \left(\frac{\mu(Q)}{|Q|} \right)^{\frac{n}{n - 1}}|Q|
&\lesssim \int_Q \prod_{j = 1}^{n} [M_j f_{j}(x)]^{\frac{1}{n - 1}} \, \d x.
\end{align*}
Summing over $Q\in\calD_m$ and applying \Cref{thm:guth} completes the proof.
\end{proof}

\begin{theorem}
	\label{thm:alberti-dimension}
	Let $(X,d,\mu)$ be a metric measure space with $n$ independent Alberti representations.  Then $\Hd \mu \geq n$.
\end{theorem}

\begin{proof}
	Suppose that $\phi \colon X \to \R^n$ is Lipschitz and $C_1,\ldots,C_n \subset \R^{n}$ are independent cones such that $\mu$ has a $C_j$-Alberti representation with respect to $\phi$ for each $1\leq j \leq n$.
	Then $\phi_{\ast}\mu$ also has a $C_j$-Alberti representation for each $1\leq j \leq n$. To see this, it suffices by \Cref{lem:c-null} to argue that if $S \subset \R^{n}$ is $C_{j}$-null (with respect to $\mathrm{id}$), then $\phi_{\ast}\mu(S) = 0$. So, fix $1 \leq j \leq n$ and a set $S \subset \R^{n}$ which is $C_{j}$-null (with respect to $\mathrm{id}$). We will show that $\phi^{-1}(S)$ is $C_{j}$-null (with respect to $\phi$): then $\phi_{\ast}\mu(S) = \mu(\phi^{-1}(S)) = 0$, since $\mu$ has a $C_{j}$-Alberti representation. Let $\gamma\in\Gamma(X)$ be a $C_j$-curve (with respect to $\phi$). Then it follows from \Cref{def:curve-direction} that $\phi \colon \im \gamma \to \R^{n}$ is biLipschitz, and hence
	\[\phi_*(\H^1\llcorner \gamma) \ll \H^1\llcorner \phi(\gamma).\]
	Also, $\phi \circ \gamma \colon \dom \gamma \to \R^{n}$ is a $C_{j}$-curve (with respect to $\mathrm{id}$), so $\mathcal{H}^{1}(S \cap \phi(\gamma)) = 0$, and consequently $\mathcal{H}^{1}(\gamma \cap \phi^{-1}(S)) = \phi_*(\H^1\llcorner \gamma)(S)=0$. This means that $\phi^{-1}(S)$ is $C_{j}$-null (with respect to $\phi$), as desired.
	
	The previous discussion shows that it suffices to prove the result for $X=\R^n$, because $\phi_{\ast}\mu$ is a measure on $\R^{n}$ satisfying the same hypotheses as $\mu$. So, suppose then that $\mu$ is a measure on $\R^{n}$ which has Alberti representations with respect to independent cones $C(w_i,\theta_i)$ for $1\leq i \leq n$.
	Then there exists a $\tau>0$ such that the cones $C(w_i,\theta_i+\tau)$ are also independent.
	Further, there exists an $L\geq 1$, depending on the $w_{i}$ and $\theta$, such that, for any choice of $v_i \in \mathbb S^{n-1}\cap C(w_i,\theta_i+\tau)$ for each $1\leq i \leq n$, there exists a linear mapping $T \colon \R^n \to \R^n$ of norm at most $L$ that maps each $v_i$ to $e_i$.
	
	Fix $\epsilon>0$ and let $\delta>0$ be given by \Cref{prop:apply-guth}.
	For any Borel $Y\subset \R^n$ with $\mu(Y)>0$, we will show that $\Hd Y \geq n-\epsilon(n-1)$.
	Since $\epsilon>0$ is arbitrary, this will complete the proof.
	By reducing $\delta$ if necessary, we may suppose that $\delta<\tau$.
	By applying \Cref{refine}, we may find a Borel set $A\subset Y$ with $\mu(A)>0$ and cones $C(v_i,\tau/L)\subset C(w_i,\theta_i+\delta)$ such that $\mu\llcorner A$ has a $C(v_i,\tau/L)$-Alberti representation for each $1\leq i \leq n$.
	Then $T_*(\mu\llcorner A)$ has a $C(e_i,\tau)$-Alberti representation for each $1\leq i \leq n$. By the Radon-Nikodym theorem, there exists $B \subset A$ of positive measure for which $T_*(\mu\llcorner B)$ satisfies the hypotheses of \Cref{prop:apply-guth} for some $M_{1},\ldots,M_{n} \leq M < \infty$. By translating, we may also assume that $T_{\ast}(\mu\llcorner B)$ charges $[0,1)^{n}$. Applying the lemma, it then remains to prove that for any measure $\mu$ on $\R^{n}$ satisfying \cref{mainMult}, and any Borel $Y\subset [0,1)^{n}$ with $\mu(Y)>0$, $\Hd Y \geq n-\epsilon(n-1)$.
	
	This can be seen by estimating the $L^{\frac{n}{n-1}}$ dimension of $\mu\llcorner [0,1)^{n}$.
	By definition this quantity is
	\[d :=(n-1)\liminf_{m\to\infty}\frac{\log\sum_{Q\in\calD_{m}} \mu(Q)^{\frac{n}{n-1}}}{\log 2^{-m}}.\]
	For any $Q\in \calD_{m}$, $|Q|/|Q|^{\frac{n}{n-1}} = 2^{\frac{mn}{n-1}}$ and so \cref{mainMult} gives
	\begin{equation*}
		\frac{\log\sum_{Q\in\calD_{m}} \mu(Q)^{\frac{n}{n-1}}}{\log 2^{-m}} \geq \frac{\log M^{\frac{n}{n-1}} C_{n,\epsilon}2^{-m(\frac{n}{n-1}-\epsilon)}}{\log 2^{-m}} \to \frac{n}{n-1} -\epsilon
	\end{equation*}
	as $m\to \infty$, so that $d\geq n-\epsilon(n-1)$.
	Since $\mu\llcorner[0,1)^{n}$ is finite and non-zero, the first inequality of \cite[Theorem 1.4]{MR1897575} states that the $L^{\frac{n}{n-1}}$ dimension of $\mu$ is at most the lower Hausdorff dimension of $\mu$. Hence $\Hd \mu\llcorner [0,1)^{n} \geq d \geq n - \epsilon(n - 1)$, and the proof is complete. \end{proof}

\section{Conclusion}\label{conclusion}

Combining our previous results gives the following.
\begin{theorem}
	\label{thm:main}
	Let $(X,d,\mu)$ be a metric measure space with $\Hd X<\infty$.
	Suppose that $\mu(S \cap \ind(\phi))=0$ for every $n\in \N$, every Lipschitz $\phi\colon X \to \R^n$ and every $S\in \mathcal{S}(\phi)$.
	Then $(X,d,\mu)$ is a Lipschitz differentiability space.
	Further, one can take $n_{i} \leq \Hd X$ for all $i \in \N$ in \Cref{lipDiffSpace}.
\end{theorem}

\begin{proof}
	By \Cref{prop:chart-decomp} it suffices to show that, for any $n > \dim_{\mathrm{H}} X$ and any $\phi\colon X \to\R^n$, $\mu(\ind\phi)=0$.
	To prove the contrapositive, suppose that $\phi\colon X \to \R^n$ is Lipschitz with $\mu(\ind\phi)>0$.
	By \Cref{alberti-decomposition}, there exists a decomposition
	$\ind\phi = S \cup \bigcup_{j = 1}^{\infty} A_{i}$
	where $S\in \mathcal{S}(\phi)$ and each $\mu\llcorner A_i$ has $n$ independent Alberti representations.
	By assumption $\mu(S)=0$ and so $\mu(A_i)>0$ for some $i \in \N$.
	\Cref{thm:alberti-dimension} implies that $\Hd X \geq \Hd A_i \geq n$, as required.
\end{proof}

We say that a Borel $\rho\colon X \to \R$ is a \emph{*-upper gradient}\footnote{
	*-upper gradients were introduced in \cite{MR3818093} using a different, but equivalent, definition.}
of a Lipschitz $f\colon X \to \R$ if, for every 1-Lipschitz $\gamma \in\Gamma(X)$,
\begin{equation*}\label{starug}
|(f\circ\gamma)'(t)| \leq \rho(\gamma(t)) \quad \L^{1}\text{-a.e.\ } t\in\dom\gamma.
\end{equation*}
If $\gamma\colon [0,l]\to X$ is 1-Lipschitz, by applying \cite[Lemma 4]{kirchheim-regularity}, we may cover almost all of $[0,l]$ by countably many compact sets $K_i$ such that each $\gamma \llcorner K_i$ is 1-Lipschitz and 2-biLipschitz, and hence belongs to $\Gamma(X)$.
If $\rho$ is a *-upper gradient of $f$ then
\[|(f\circ\gamma)'(t)|=|(f\circ\gamma \llcorner K_i)'(t)| \leq \rho(\gamma(t)) \quad \L^1 \text{-a.e.\ } t\in K_i,\ \forall i\in\N.\]
Therefore, by the fundamental theorem of calculus, any *-upper gradient is an upper gradient.
However, because *-upper gradients only consider curve fragments, the converse is not true in general (consider $X\subset \R$ a Cantor set of positive Lebesgue measure).
Thus the following theorem implies \Cref{main}.
\begin{theorem}\label{main-star} Let $(X,d,\mu)$ be a metric measure space with $\Hd X <\infty$.
	Suppose that there exists a $C\geq 1$ such that, for every Lipschitz $f\colon X \to \R$ and every *-upper gradient $\rho$ of $f$,
	\begin{equation}\label{Lipug-star}\Lip(f,x) \leq C\rho(x) \quad \mu \text{-a.e.\ } x\in X.\end{equation}
	Then $(X,d,\mu)$ is a Lipschitz differentiability space. Further, one can take $n_{i} \leq \Hd X$ for all $i \in \N$ in \Cref{lipDiffSpace}. \end{theorem}

\Cref{main-star} follows by combining \Cref{thm:main} with the following.
\begin{proposition}\label{lem:c-null-in-PI}
  Let $(X,d,\mu)$ be a metric measure space and suppose that there exists a $C\geq 1$ such that, for any Lipschitz $f\colon X \to \R$ and any *-upper gradient $\rho$ of $f$, \eqref{Lipug-star} holds.
  Then for every $n\in\N$, every Lipschitz $\phi\colon X \to \R^{n}$ and every $S\in\mathcal S(\phi)$, $\mu(S\cap\ind(\phi))=0$.
\end{proposition}

\begin{proof}
	For each $x\in \ind\phi$, the function
	\[ D\in \S \mapsto \Lip(D\cdot\phi,x)\]
	is Lipschitz (this follows from $\Lip(f + g,x) \leq \Lip(f,x) + \Lip(g,x)$) and hence bounded away from $0$. Let $\lambda_x>0$ be its minimum value. Now, it suffices to show that any $\mathcal{S}(\phi)$ subset of $U_i =\{x \in \ind \phi :\lambda_x \geq \tfrac{1}{i}\}$, $i\in\N$, for $i \in \N$ fixed, is $\mu$-null.
		
	Fix a $\mathcal{S}(\phi)$ subset $S \subset U_{i}$. Then, fix $\theta = \theta(i) > 0$ (to be determined momentarily), and decompose $S = S_{1} \cup \ldots \cup S_{m}$, where each $S_{j}$ is $C(w_{j},\theta)$-null for some $w_{j} \in \S$. It suffices to show that $\mu(S_{j}) = 0$ for each $1 \leq j \leq m$, so we may assume to begin with that $S$ is $C(w,\theta)$-null for some $w \in \S$. Now, define $f := w\cdot \phi \colon X \to \R$.
	We claim that
	\[\rho = \Lip\phi(\theta \chi_{S}+ \chi_{S^c})\]
	is a *-upper gradient of $f$.
	Indeed, if $\gamma\in\Gamma(X)$ is 1-Lipschitz, then by \Cref{lem:cone-null}, $(\phi\circ\gamma)'(t) \not\in \inti(C(w,\theta))$ for almost every $t\in \gamma^{-1}(S)$.
	That is, for almost every $t\in \gamma^{-1}(S)$,
	\begin{equation}\label{rho-is-*ug}
	|(f\circ\gamma)'(t)|=|w\cdot(\phi\circ\gamma)'(t)|\leq \theta\|(\phi\circ\gamma)'(t)\| \leq \theta \Lip \phi= \rho(\gamma(t)).
	\end{equation}
        Of course, for almost every $t\in\gamma^{-1}(S^{c})$, $|(f\circ\gamma)'(t)| \leq \Lip\phi$.
	Therefore, by \eqref{Lipug-star}, for $\mu$ almost every $x\in S$,
	\begin{equation*}
	1/i \leq \Lip(f,x) \leq C \Lip\phi \cdot \theta.
	\end{equation*}
	This shows that $\mu(S) = 0$ for $\theta = \theta(i) > 0$ sufficiently small, and the proof is complete.
\end{proof}

\begin{remark} Note that the proof of \Cref{lem:c-null-in-PI} did not use the full strength of the hypothesis \cref{Lipug-star}. It would, instead, suffice to assume the following. For all $x \in X$ there exists a ``modulus of continuity" $\omega_{x} \colon (0,\infty) \to (0,\infty)$ with $\omega_{x}(r) \searrow 0$ as $r \searrow 0$ such that $\Lip(f,x) \leq \omega_{x}(\rho(x))$ for all Lipschitz $f \colon X \to \R$, for all *-upper gradients $\rho$ of $f$, and for $\mu$ almost every $x \in X$. \end{remark}

\end{document}